\documentclass[12pt,3p]{nj_style}

\usepackage{amsmath}
\usepackage{amssymb}
\usepackage{amsthm}
\usepackage{enumerate}
\usepackage[font=small,labelfont=bf,margin=8pt]{caption}
\usepackage[usenames,dvipsnames]{color}
\usepackage{graphicx}
\usepackage[colorlinks=true]{hyperref}

\newtheorem{thm}{Theorem} 

\newtheorem{conj}[thm]{Conjecture}

\newtheorem{lemma}[thm]{Lemma}
\newtheorem{prop}[thm]{Proposition}

\newcommand{\mnorm}[1]{%
\left\vert\kern-0.9pt\left\vert\kern-0.9pt\left\vert #1
\right\vert\kern-0.9pt\right\vert\kern-0.9pt\right\vert}
\newcommand{\bmnorm}[1]{%
\big\vert\kern-0.9pt\big\vert\kern-0.9pt\big\vert #1
\big\vert\kern-0.9pt\big\vert\kern-0.9pt\big\vert}

\begin{document}

\begin{frontmatter}



\title{Non-Uniqueness of Minimal Superpermutations}

\author[IQC,UG]{Nathaniel Johnston}
\ead{nathaniel.johnston@uwaterloo.ca}

\address[IQC]{Institute for Quantum Computing, University of Waterloo, Waterloo, Ontario N2L~3G1, Canada}
\address[UG]{Department of Mathematics \& Statistics, University of Guelph, Guelph, Ontario N1G~2W1, Canada}

\begin{abstract}
	We examine the open problem of finding the shortest string that contains each of the $n!$ permutations of $n$ symbols as contiguous substrings (i.e., the shortest superpermutation on $n$ symbols). It has been conjectured that the shortest superpermutation has length $\sum_{k=1}^n k!$ and that this string is unique up to relabelling of the symbols. We provide a construction of short superpermutations that shows that, if the conjectured minimal length is true, then uniqueness fails for all $n \geq 5$. Furthermore, uniqueness fails spectacularly; we construct more than doubly-exponentially many distinct superpermutations of the conjectured minimal length.
\end{abstract}

\begin{keyword}
superpermutation \sep permutation \sep shortest superstring

\end{keyword}

\end{frontmatter}

\section{Introduction}\label{sec:intro}

The shortest superstring problem asks for the shortest string containing each string $s_i$ in a given set $S := \{s_1,s_2,\ldots,s_m\}$ as a contiguous substring. The shortest superstring problem is NP-hard~\cite{GMS80} in general, and finding algorithms for computing or approximating shortest superstrings is an active area of research \cite{GL01,KS05}.

In certain special cases, the shortest superstring problem simplifies considerably. Perhaps most well-known of these cases arises when the set of substrings $S$ consists of all words of length $k$ on $n$ symbols, where $k$ and $n$ are given positive integers. There are $n^k$ such words, so there is a trivial lower bound of length $n^k + k - 1$ on any string containing them all as substrings. More interestingly, however, it has been shown \cite{AB51} that this length is always attainable. A string containing all such words that attains this minimal length is called a \emph{De Bruijn sequence}.

In the present paper, we consider the related problem that arises when the set of substrings $S$ contains all permutations of $n$ symbols. That is, we are interested in the problem that asks for the shortest string on the $n$ symbols $[n] := \{1,2,3,\ldots,n\}$ that contains all $n!$ permutations of those symbols as contiguous substrings. Following the terminology of \cite{AT93}, we call a string containing all permutations in this way a \emph{superpermutation}, and we call a superpermutation of minimum length a \emph{minimal superpermutation}.

Similar to the De Bruijn sequence case, there are $n!$ permutations of $n$ symbols, so a trivial lower bound on the length of a minimal superpermutation is $n! + n - 1$. Unlike in the De Bruijn sequence case, however, this bound is not tight when $n \geq 3$. For $n = 2$ a minimal superpermutation is trivially $121$, which contains both permutations $12$ and $21$ as substrings. For $n = 3$ it is not difficult to show that a shortest string containing each of $123$, $132$, $213$, $231$, $312$, and $321$ as contiguous substrings is $123121321$. For small values of $n$, the optimal solutions have been computed via brute-force computer search \cite{A180632} --- Table~\ref{table:seq_begin} summarizes the shortest strings for $1 \leq n \leq 4$.
\begin{table}[ht]
\centerline{
	\begin{tabular}{ c | l | c }
	  $n$ & Minimal Superpermutation & Length \\\hline
	  $1$ & $1$ & $1$ \\
	  $2$ & $121$ & $3$ \\
	  $3$ & $123121321$ & $9$ \\
	  $4$ & $123412314231243121342132413214321$ & $33$
	\end{tabular}}
	\caption{The shortest strings containing every permutation of $\{1,2,\ldots,n\}$ for $1 \leq n \leq 4$.}\label{table:seq_begin}
\end{table}

Computing minimal superpermutations for $n \geq 5$ isn't feasible via brute-force, but the solutions for $n \leq 4$ provide some structure that is easily generalized. For example, there is a natural recursive construction that can be used to create small superpermutations on $n$ symbols for any $n$ (more specifically, superpermutations of length $\sum_{k=1}^n k!$), and this procedure produces the superpermutations provided in Table~\ref{table:seq_begin} when $n \leq 4$. This method of constructing small superpermutations has been rediscovered independently in several different forms \cite{AT93,JBweb,JOweb,YFweb,kolistivra} and it gives, for all $n$, what are currently the only known superpermutations of length at most $\sum_{k=1}^n k!$. Furthermore, the superpermutations provided in Table~\ref{table:seq_begin} have been shown by brute force to be unique up to relabelling the symbols (e.g., the only minimal superpermutations in the $n = 2$ case are $121$ and $212$, which are related to each other by swapping the roles of $1$ and $2$).

These observations have led to the following conjectures:
\begin{conj}[\cite{AT93,A180632,JBweb,JOweb,YFweb,JBweb2}]\label{conj:min_length}
	The length of the minimal superpermutation on $n$ symbols is $\sum_{k=1}^n k!$.
\end{conj}

\begin{conj}[\cite{AT93}]\label{conj:unique}
	The minimal superpermutation on $n$ symbols is unique up to interchanging the roles of the symbols.
\end{conj}

In this work we generalize the well-known method of construction of small superpermutations in a new way. Our construction shows that, for all $n \geq 5$, at least one of these two conjectures fails. It is known that Conjecture~\ref{conj:min_length} has to at least be ``almost'' true in the sense that the conjectured length can't be too far from the true minimal length \cite[Theorem~18]{AT93}. On the other hand, our construction shows that if Conjecture~\ref{conj:min_length} is true, then Conjecture~\ref{conj:unique} fails spectacularly as $n$ increases. For example, our construction in the $n = 8$ case provides more than $3 \times 10^{50}$ distinct superpermutations of the conjectured minimal length.

Before proceeding, we present the notational conventions and representations of permutations that we will use throughout this paper. In order to avoid double-counting superpermutations that are related to each other by interchanging the roles of the symbols, we only consider superpermutations that begin with the substring $12\cdots n$. We typically represent a permutation $\sigma \in S_n$ via its one-line representation $\sigma(1)\sigma(2)\cdots\sigma(n)$, which we think of as a string (for example, the identity permutation corresponds to the string $12\cdots n$).

It is also sometimes useful for us to represent a permutation $\sigma$ by another encoding that we now describe. Define $p_k \in S_n$ to be the permutation that cyclically increases the first $k$ positive integers by one and acts as the identity otherwise: $p_k(i) = i+1$ if $1 \leq i < k$, $p_k(i) = i$ if $k < i \leq n$, and $p_k(k) = 1$. Then $p_k^j$ cyclically increases the first $k$ positive integers by $j$ and acts as the identity otherwise, and for every permutation $\sigma \in S_n$, there is a unique tuple of exponents $(j_2, j_{3}, \ldots, j_n)$ with $0 \leq j_i < i$ for all $i$ so that $\sigma = p_2^{j_2} \circ p_{3}^{j_{3}} \circ \cdots \circ p_n^{j_n}$. To write $\sigma$ in this way, first choose $j_n$ so that the correct integer is mapped to $n$, then choose $j_{n-1}$ so that the correct integer is mapped to $n-1$, and so on. To see uniqueness, we simply note that there are exactly $n!$ tuples $(j_2, j_{3}, \ldots, j_n)$ satisfying $0 \leq j_i < i$ for all $i$, which is the same as the size of $S_n$.

Rather than explicitly writing a permutation in the form $p_2^{j_2} \circ p_{3}^{j_{3}} \circ \cdots \circ p_n^{j_n}$, we simply write the exponents in order within square brackets and a subscript $c$, as in $[j_2 j_{3} \cdots j_n]_c$. We call this representation the \emph{circular shift representation}. For example, both $42351$ (one-line notation) and $[0121]_c$ (circular shift representation) represent the same permutation.

As perhaps a slightly more intuitive way of converting a permutation's circular shift representation to its one-line representation, start with the string $12\cdots n$. Cycle the first $2$ characters of the string to the left the number of times indicated by $j_2$, then cycle the first $3$ characters of the string to the left the number of times indicated by $j_3$, and so on. So to obtain the one-line representation of $[0121]_c$, we start with $12345$ and transform it as follows:
\begin{align*}
	12345 \stackrel{0}{\rightarrow} 12345 \stackrel{1}{\rightarrow} 23145 \stackrel{2}{\rightarrow} 14235 \stackrel{1}{\rightarrow} 42351.
\end{align*}

\section{Construction of a Small Superpermutation}\label{sec:construct}

In this section we briefly present a well-known method for constructing a family of strings $\{M_1, M_2, M_3, \ldots\}$ with the property that $M_n$ is a superpermutation on $[n]$ of length $|M_n| = \sum_{k=1}^n k!$.

First, let $M_1 = 1$, the string with just one symbol. Then we construct $M_{n+1}$ from $M_n$ via the following procedure:
\begin{enumerate}[(a)]
	\item Let $P_j$ ($0 \leq j < n!$) be the $j$th permutation of $[n]$ that appears as a substring of $M_n$.
	\item Let $Q_j$ ($0 \leq j < n!$) be the string of length $2n+1$ obtained by concatenating the string $P_j$, the symbol $n+1$, and again the string $P_j$, in that order.
	\item Concatenate $Q_0, Q_1, \ldots, Q_{n!-1}$, in that order, overlapping as much as possible each consecutive pair of strings.
\end{enumerate}
To demonstrate this procedure in the $n = 2$ case, start with $M_2 = 121$. Then $P_0 = 12$ and $P_1 = 21$, so $Q_0 = 12312$ and $Q_1 = 21321$. If we concatenate $Q_0$ and $Q_1$, overlapping as much as possible, we get $M_3 = 123121321$, where we overlapped the $2$ in the middle of the string. The procedure is illustrated in the $n = 3$ case in Figure~\ref{fig:s3tree}.
\begin{figure}[ht]
	\centerline{\includegraphics[width=\textwidth]{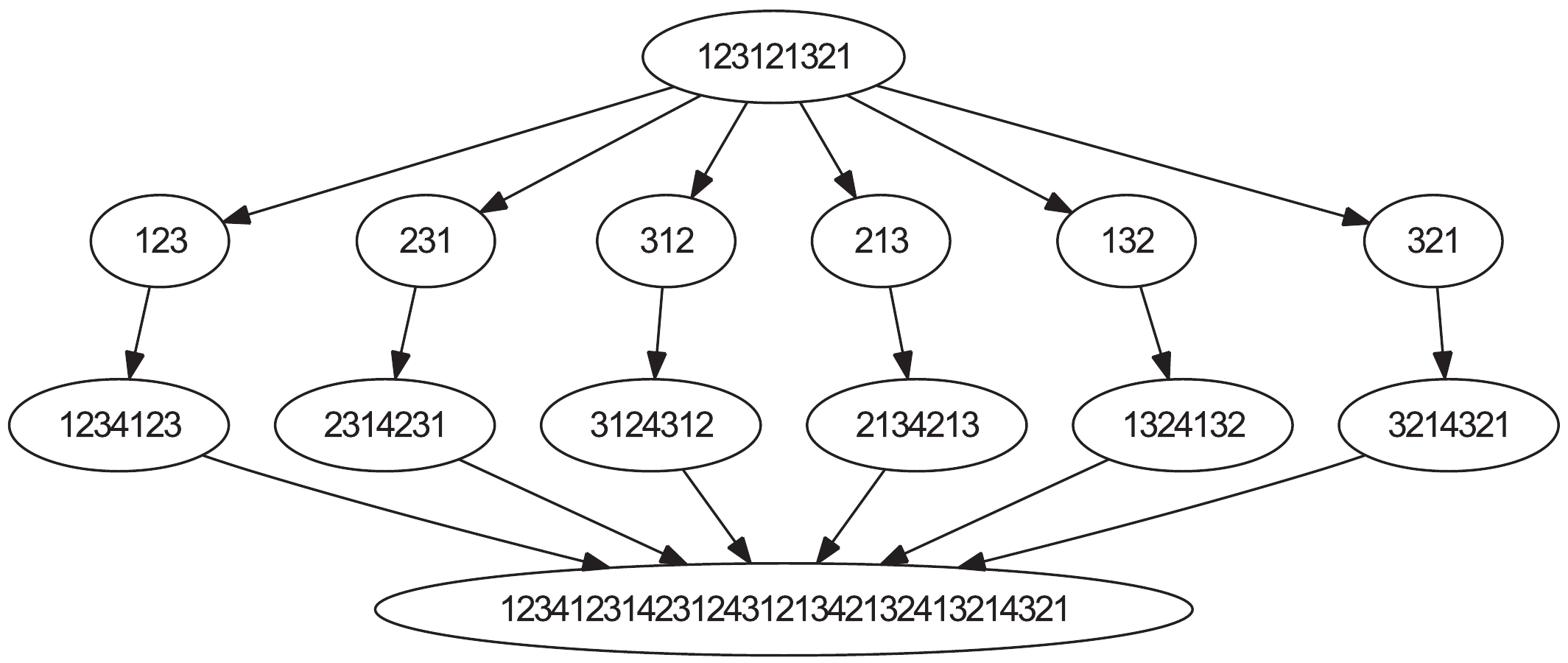}}
	\caption{A tree that demonstrates how to construct $M_4$ from $M_3$.}\label{fig:s3tree}
\end{figure}

The fact that each $M_n$ is a superpermutation on $[n]$ with $|M_n| = \sum_{k=1}^n k!$ follows easily via induction, so we don't present the argument here. The intuition behind why these superpermutations are so efficient is simple --- they are the ``greedy'' superpermutations obtained by starting with the string $12\cdots n$ and repeatedly appending as few symbols as possible so as to add a new permutation as a contiguous substring. However, minimality and uniqueness have only been proved when $n \leq 4$.

It is interesting to observe the circular shift representation of all permutations in the order that they appear in $M_n$ (and indeed, this observation is our primary reason for introducing circular shift representations in the first place). For example, when $n = 3$, the permutations appear in $M_n$ in the order $123$, $231$, $312$, $213$, $132$, $321$; these have circular shift representations $[00]_c$, $[01]_c$, $[02]_c$, $[10]_c$, $[11]_c$, and $[12]_c$, respectively. In other words, the order of the permutations in $M_3$ is the same as the order we get if we ``count'' in the circular shift representation. It turns out that this is true in general, as demonstrated by the following proposition.
\begin{prop}\label{prop:circ_shift_ordering}
	Let $n$ and $j$ be positive integers with $0 \leq j < n!$. If we write
	\begin{align*}
		j = \sum_{i=2}^n j_{i} \cdot n!/i!
	\end{align*}
	with $0 \leq j_i < i$ for all $i$ then the $(j+1)$th permutation to appear in $M_n$ as a contiguous substring is $[j_2 j_{3} \cdots j_n]_c$.
\end{prop}
Before proving Proposition~\ref{prop:circ_shift_ordering}, we note that every integer $0 \leq j < n!$ can be written (uniquely) in the desired form --- this is just a slight variant on the factorial number system \cite{Knu73} in which the digits are allowed to increase from left to right, rather than from right to left.
\begin{proof}[Proof of Proposition~\ref{prop:circ_shift_ordering}]
	We prove the result by induction. We already noted that the result is true when $n = 3$, and it is also easily verified when $n = 1$ or $n = 2$. We thus move directly to the inductive step.
	
	Suppose that the result holds for a given value of $n$. Then in the construction of $M_{n+1}$ given earlier, $P_j$ is the $(j+1)$th permutation of $n$ symbols that appears in $M_n$, so $P_j = [j_2 j_{3} \cdots j_n]_c$ by the inductive hypothesis. Then $Q_j$ contains $n+1$ permutations of $[n+1]$ as substrings, the first of which is $[j_2 j_{3} \cdots j_n 0]_c$. The second permutation contained as a substring of $Q_j$ is just a cyclic permutation of the first: $[j_2 j_{3} \cdots j_n 0]_c \circ p_{n+1} = [j_2 j_{3} \cdots j_n 1]_c$. Similarly, the third permutation that appears as a substring of $Q_j$ is $[j_2 j_{3} \cdots j_n 1]_c \circ p_{n+1} = [j_2 j_{3} \cdots j_n 2]_c$, and so on up to $[j_2 j_{3} \cdots j_n n]_c$. The result follows.
\end{proof}

\section{Non-Uniqueness When \texorpdfstring{$n = 5$}{n = 5}}\label{sec:n5construct}

In Section~\ref{sec:gen_construct} we present our generalized construction of small superpermutations for arbitrary $n$, but for now we simply present and motivate it in the first non-trivial case (i.e., when $n = 5$) for clarity.

To start, notice that the first $60$ permutations to appear as substrings of $M_5$ are those of the form $[0xyz]_c$ for some $0 \leq x < 3$, $0 \leq y < 4$, and $0 \leq z < 5$, while the last $60$ permutations to appear as substrings of $M_5$ are those of the form $[1xyz]_c$. Phrased differently, the first $60$ permutations are those in which the symbols $1$ and $2$ have not been transposed. In other words, when written in one-line notation, if the $4$ and $5$ are erased, we are left with $123$, $231$, or $312$. Similarly, the last $60$ permutations to appear in $M_5$ are those that result in $213$, $132$, or $321$ when the $4$ and $5$ are removed from their one-line representation. Since this property has nothing to do with the order of $4$ and $5$ in the permutation's one-line representation, we conclude that a permutation is one of the first (last) $60$ permutations to appear as a substring of $M_5$ if and only if interchanging the roles of $4$ and $5$ in its one-line representation results again in one of the first (last) $60$ permutations to appear in $M_5$.

This leads immediately to the key idea: we can just replace all $4$'s by $5$'s (and vice-versa) in the last half of $M_5$ to generate a new string $M_5^\prime$ that is of the same length and is still a superpermutation. These two strings are written in full in Table~\ref{table:M5_and_new}, and it is easily verified by computer that they are indeed distinct superpermutations of length $\sum_{k=1}^5 k! = 153$.
\begin{table}[ht]
\centerline{
	\begin{tabular}{ l | l }
	  $M_5$ & $1234512341523412534123541231452314253142351423154231\cdot$ \\
& $\quad\quad\quad\cdot 245312435124315243125431\mathbf{2}134521342513421534213542\cdot$ \\
& $\quad\quad\quad\quad\quad\quad\cdot 1324513241532413524132541321453214352143251432154321$
 \\\hline
	  $M_5^\prime$ & $1234512341523412534123541231452314253142351423154231\cdot$ \\
& $\quad\quad\quad\cdot 245312435124315243125431\mathbf{2}135421352413521435213452\cdot$ \\
& $\quad\quad\quad\quad\quad\quad\cdot 1325413251432513425132451321543215342153241532145321$
	\end{tabular}}
	\caption{Two distinct superpermutations on $[5]$ of length $\sum_{k=1}^5 k! = 153$. The second string is obtained from the first simply by interchanging the roles of the symbols $4$ and $5$ in the second half of the string (i.e., after the bold \textbf{2}).}\label{table:M5_and_new}
\end{table}

\section{Generalized Construction}\label{sec:gen_construct}

We now show how the idea from Section~\ref{sec:n5construct} extends to the $n \geq 6$ case. While the argument only provides two distinct superpermutations of the conjectured minimal length in the $n = 5$ case, our main result shows that the number of small superpermutations produced grows extremely quickly as $n$ increases.
\begin{thm}\label{thm:non_unique}
	There are at least
	\begin{align}\label{eq:non_unique_formula}
		\prod_{k=1}^{n-4} (n-k-2)!^{k\cdot k!}.
	\end{align}
	distinct (up to relabelling of the symbols) superpermutations on $[n]$ of length $\sum_{k=1}^n k!$.
\end{thm}

Before proving the theorem, note that the formula~\eqref{eq:non_unique_formula} gives the empty product when $n \leq 4$, which agrees with the previously-mentioned fact that the minimal superpermutation is unique in these cases. For $n = 5, 6, 7, 8$ it gives the values $2, 96, 8153726976,$ and approximately $3 \times 10^{50}$ respectively. A text file containing all $96$ small superpermutations described by the theorem in the $n = 6$ case can be downloaded from \cite{SuperPerm6}. A text file containing all such superpermutations in the $n = 7$ case would be larger than $43$ terabytes in size and is thus omitted.

The following lemma is our main tool in the proof of the theorem. Before presenting the lemma, we note that each permutation appears in $M_n$ exactly once \cite[Lemma~4]{AT93}. It follows that the string $T_{j,k}$ defined in the statement of the lemma is unique.
\begin{lemma}\label{lem:generalize_helper}
	Let $2 \leq k < n$ be integers. For each $0 \leq j < k!$, let $T_{j,k}$ be the shortest substring of $M_n$ that contains each of the $(j\cdot n!/k!+1)$th through $((j+1)\cdot n!/k!)$th permutations to appear in $M_n$ as substrings. Then for all $j$,
	\begin{enumerate}[(a)]
		\item there exists $1 \leq \ell < k$ such that the last $\ell$ characters of $T_{j,k}$, in order, are the same as the first $\ell$ characters of $T_{j+1,k}$;
		\item the first and last $k+1$ characters of $T_{j,k}$ are $1, 2, \ldots, k$, and $k+1$ (not necessarily in that order); and
		\item the set of permutations that are substrings of $T_{j,k}$ remains unchanged under any interchange of the roles of the symbols $k+2, k+3, \ldots, n$ in $T_{j,k}$.
	\end{enumerate}
\end{lemma}
\begin{proof}
	We first prove properties (a) and (b) by induction. Both of these properties are clearly true by inspection in the $n = 3$ case. For the inductive step, assume that the result holds for $M_n$. If we fix $k$ then properties (a) and (b) clearly still hold for $M_{n+1}$ by the construction given in Section~\ref{sec:construct}. To see that properties (a) and (b) also hold for $M_{n+1}$ when $k = n$, simply note that we have $T_{j,k} = Q_j$ in this case, and $Q_j$ clearly has both properties.
	
	We now prove property (c) (directly; not via induction). It is a direct consequence of Proposition~\ref{prop:circ_shift_ordering} that there exist fixed $j_2, j_3, \ldots, j_{k}$ such that $T_{j,k}$ contains as substrings exactly the permutations whose cyclic shift representations are of the form $[j_2 j_3 \cdots j_n]_c$, where $j_i$ is free to vary from $0$ to $i-1$ when $k+1 \leq i \leq n$. It follows that there exists a fixed permutation $\tau_j$ on $[k+1]$ such that $\sigma \in S_n$ is a substring of $T_{j,k}$ if and only if deleting the symbols $k+2, k+3, \ldots, n$ from its one-line representation results in a cyclic permutation of $\tau_j$. Since this criterion for being a substring of $T_{j,k}$ is clearly independent of the locations of the symbols $k+2, k+3, \ldots, n$ in the one-line representation of $\sigma$, condition (c) follows.
\end{proof}

\begin{proof}[Proof of Theorem~\ref{thm:non_unique}]
Our procedure is to recursively construct the given number of small superpermutations from the well-known superpermutation $M_n$. Observe that the procedure that we are about to describe reduces exactly to the procedure described in Section~\ref{sec:n5construct} when $n = 5$.

\begin{enumerate}[(1)]
	\item Let $k = n - 3$ and define $SP := \{M_n\}$ to be the set of all superpermutations of the desired length that we have constructed so far.
	
	\item For each $T \in SP$ and each $0 \leq j < k!$, let $T_{j,k}$ be as in Lemma~\ref{lem:generalize_helper}. By the lemma, we can interchange the roles of $k+2, k+3, \ldots, n$ within the section of $T$ that equals $T_{j,k}$ without affecting whether or not $T$ is a superpermutation. We could do this for each value of $j$ from $0$ to $k!-1$, however we only do it for the $k! - k$ values of $j$ that are not multiples of $(k-1)!$ (we will explain the reason for this shortly). Add all of the superpermutations constructed in this way to $SP$. The cardinality of $SP$ is thus multiplied by $(n-k-1)!^{k!-k}$ during this step.
	
	\item Decrease $k$ by $1$.
	
	\item If $k = 1$, stop. Otherwise, return to step (2).
\end{enumerate}

It is perhaps not clear in step (2) why we only consider the values of $j$ that are not divisible by $(k-1)!$. The reason why we avoid $j = 0$ is straightforward: we want all of our superpermutations to start with the substring $12\cdots n$. The reason why we avoid all other values of $j$ that are divisible by $(k-1)!$ is that if we swapped the roles of $k+2, k+3, \ldots, n$ in the same way for each $j$ from $i\cdot(k-1)!$ to $(i+1)\cdot(k-1)!-1$ for some $i$, then we would end up with a superpermutation that is the same as one produced in the next iteration of the algorithm (i.e., after we decreased $k$ by $1$). We would thus end up double-counting some superpermutations. As is, in each iteration of the algorithm we alter part of each string $T$ that has never been altered before, so we can be sure that each step in the procedure really does produce superpermutations that are not already in the set $SP$.

To get the cardinality of the set $SP$ after the algorithm terminates (i.e., the number of distinct superpermutations that we have constructed), we simply recall that its cardinality is increased by a factor of $(n-k-1)!^{k!-k}$ for each value of $k$ from $n-3$ down to $2$. That is, the number of distinct superpermutations that we have constructed is
\begin{align*}
	\prod_{k=2}^{n-3} (n-k-1)!^{k!-k},
\end{align*}
which can be seen to equal the value given in the statement of the theorem by shifting the value of $k$ by $1$.
\end{proof}

\section{Conclusions}\label{sec:conclude}

We have shown that the standard construction of small superpermutations generalizes in a natural way that disproves at least one of two long-standing conjectures. While it is possible that Conjecture~\ref{conj:min_length} is false and Conjecture~\ref{conj:unique} is true, we consider this possibility rather unlikely, as Conjecture~\ref{conj:min_length} being false would seem to indicate that minimal superpermutations in general do not have as much structure as suggested by the $n \leq 4$ cases. We thus expect that either both conjectures fail, or Conjecture~\ref{conj:min_length} is true and Conjecture~\ref{conj:unique} is false.

Our results also shed some light on approaches to proving Conjecture~\ref{conj:min_length} that will \emph{not} work. Upon seeing the abundance of symmetry in the superpermutations in the $1 \leq n \leq 4$ cases, it is tempting to try to prove that minimal superpermutations must have certain properties that the $M_n$ have. For example, one could try to prove that a minimal superpermutation on $[n]$ must be a palindrome or must contain exactly $(n-1)!$ occurrences of the symbol $n$. However, our generalized construction shows that minimal superpermutations need not have either of these properties.

While computing the length of minimal superpermutations is a fundamental combinatorial problem, and one that can be understood by anyone with even a modest mathematical background, it seems to have received relatively little attention. We hope that our results spark further interest.

\vspace{0.1in} \noindent{\bf Acknowledgements.} Thanks are extended to Dan Ashlock for helpful conversations. The author was supported by the University of Guelph Brock Scholarship, the Natural Sciences and Engineering Research Council of Canada, and the Mprime Network.

\bibliographystyle{plain}

\end{document}